\documentclass[abstract=on, 11pt, a4paper]{scrartcl}

\usepackage[T1]{fontenc}

\usepackage{amsmath, amssymb, amsthm}
\usepackage{mathtools}
\usepackage{latexsym}

\usepackage{amsbsy}

\usepackage[left=3cm, right=2.5cm, top=2.5cm, bottom=3.5cm]{geometry}

\usepackage{enumitem}
\usepackage{comment}

\usepackage{xcolor}

\usepackage{tikz}
\usetikzlibrary{circuits, intersections}
\usetikzlibrary{plotmarks}
\usetikzlibrary{angles, quotes}

\usepackage{graphicx}
\usepackage{pdfpages}

\usepackage[super]{nth}
\usepackage{hyperref}
\usepackage{todonotes}

\newcounter{countercheck}[section]

\theoremstyle{plain}
\swapnumbers
\newtheorem{theorem}[countercheck]{Theorem}
\newtheorem{proposition}[countercheck]{Proposition}
\newtheorem{lemma}[countercheck]{Lemma}
\newtheorem{corollary}[countercheck]{Corollary}

\theoremstyle{definition}

\newtheorem{convention}[countercheck]{Convention}

\theoremstyle{remark}
\newtheorem{remark}[countercheck]{Remark}

\renewcommand{\phi}{\varphi}
\renewcommand{\epsilon}{\varepsilon}

\newcommand{\NN}{\mathbb{N}}
\newcommand{\ZZ}{\mathbb{Z}}

\newcommand{\RR}{\mathbb{R}}

\newcommand{\FF}{\mathbb{F}}

\newcommand{\C}{\mathcal{C}}
\renewcommand{\P}{\mathcal{P}}

\DeclareMathOperator{\Aut}{Aut}

\DeclareMathOperator{\proj}{proj}

\numberwithin{equation}{section} 
\title{On Growth Functions of Coxeter Groups}
\author{Sebastian \textit{Bischof}\footnote{email: sebastian.bischof@math.uni-giessen.de} \\
	Mathematisches Institut, Arndtstra\ss e 2, 35392 Gie\ss en, Germany}

\date{\today}

\begin{document}

\maketitle

\begin{abstract}
	Let $(W, S)$ be a Coxeter system of rank $n$ and let $p_{(W, S)}(t)$ be its growth function. It is known that $p_{(W, S)}(q^{-1}) < \infty$ holds for all $n \leq q \in \NN$. In this paper we will show that this still holds for $q = n-1$, if $(W, S)$ is $2$-spherical. Moreover, we will prove that $p_{(W, S)}(q^{-1}) = \infty$ holds for $q = n-2$, if the Coxeter diagram of $(W, S)$ is the complete graph. These two results provide a complete characterization of the finiteness of the growth function in the case of $2$-spherical Coxeter systems with complete Coxeter diagram.
		
	
\end{abstract}

\section{Introduction}

One of the most central results in the theory of lattices is Margulis' Normal Subgroup Theorem for  irreducible lattices in connected semi-simple Lie groups of real rank $\geq 2$ with finite center and no non-trivial compact factor \cite{Ma91}. Among all the recent generalizations, let me mention that Bader and Shalom proved a version of the Normal Subgroup Theorem for irreducible cocompact lattices in a product of two locally compact, non-discrete, compactly generated groups \cite{BS06}. Based on earlier results in \cite{Re05}, Caprace and Rémy applied the Normal Subgroup Theorem to show simplicity for Kac-Moody groups over finite fields of irreducible, non-spherical and non-affine type that are twin building lattices (cf.\ \cite[Theorem 18]{CR09}). Moreover, it can be used to prove virtual simplicity of certain twin tree lattices with non-trivial commutation relations (cf.\ \cite{CR16}).

In \cite{Re99} and \cite{CG99}, Rémy, and independently Carbone and Garland, proved that certain groups acting on (twin) buildings are lattices. To be more precise: Let $(W, S)$ be a Coxeter system of finite rank and let $\Phi := \Phi(W, S)$ be its associated set of roots (viewed as half-spaces). Let $\mathcal{D} = (G, (U_{\alpha})_{\alpha \in \Phi})$ be an RGD-system of type $(W, S)$, i.e.\ a group $G$ together with a family $(U_{\alpha})_{\alpha \in \Phi}$ of subgroups (which we call \emph{root groups}) indexed by the set of roots $\Phi$ satisfying some combinatorial axioms (for the precise definition we refer to \cite[Ch.\ $7,8$]{AB08}). Then there exists a \emph{twin building} $\Delta = (\Delta_+, \Delta_-, \delta_*)$ such that $G$ acts on $\Delta$. It turns out that under some conditions, $G^{\dagger} := \langle U_{\alpha} \mid \alpha \in \Phi \rangle \leq \Aut(\Delta_+) \times \Aut(\Delta_-)$ and $U_+ := \langle U_{\alpha} \mid \alpha \in \Phi_+ \rangle \leq \Aut(\Delta_-)$ are lattices (cf. \cite{Re99}, \cite{CG99}) -- and in this case $G^{\dagger}$ is an example of a twin building lattice. Sufficient conditions are that every root group is finite, $W$ is infinite and for $q_{\min} := \min\{ \vert U_{\alpha} \vert \mid \alpha \in \Phi \}$ one has $p_{(W, S)}\left( \frac{1}{q_{\min}} \right) < \infty$, where $p_{(W, S)}(t)$ denotes the \emph{growth function} of $(W, S)$. It is clear that this is finite if $\vert S \vert \leq q_{\min}$. It is particularly unsatisfying that this criterion does not apply to Coxeter systems of rank $n \geq 3$ and $q_{\min} = 2$. However, there are examples of Coxeter systems $(W, S)$ of rank $n \geq 3$ with $p_{(W, S)}\left( \frac{1}{2} \right) < \infty$. Note that the growth function $p_{(W, S)}(t)$ applied to $q^{-1}$ with $q\in \NN$ and $q\geq 2$ is finite for spherical and affine Coxeter systems (cf.\ \cite[Ch.\ VI, Exercises $\S 4$, $10$]{Bo68}).

Suppose $(W, S)$ is of type $(4, 4, 4)$, that is, $(W, S)$ is of rank $3$ and $o(st) = 4$ for all $s\neq t\in S$. In \cite{BiDiss} we constructed uncountably many new examples of RGD-systems of type $(4, 4, 4)$ in which every root group has cardinality $2$. As the known criterion does not apply to such RGD-systems, we first asked the question whether $p_{(W, S)}\left( \frac{1}{2} \right) < \infty$ holds. It turns out that this is indeed the case and the reason is the following more general result, which only depends on the rank of $(W, S)$ and not on the type (cf.\ Theorem \ref{Theorem: convergence of growth function}):

\medskip
\noindent \textbf{Theorem A:} Let $(W, S)$ be a $2$-spherical Coxeter system of rank $n$. Then $p_{(W, S)}\left( \frac{1}{n-1} \right) < \infty$.

\medskip
After completion of this project I was informed by Corentin Bodart that a more general version of Theorem A can be deduced from \cite[Theorem $1$]{AL02} (cf.\ also \cite{dlH00}). This implies that one can replace in Theorem A \emph{$2$-spherical} by \emph{non-universal}, i.e.\ $m_{st} < \infty$ for some $s\neq t\in S$. Our methods of the proof are very different and most of the results proved in this paper are also used to prove Theorem C below. Our proofs are Coxeter group theoretic, while the proofs in \cite{AL02} are for non-elementary word hyperbolic groups.

\medskip
In view of the examples constructed in \cite{BiDiss}, Theorem A produces many new examples of lattices in (locally compact) automorphism groups of buildings and in a product of two automorphism groups of buildings. Combining Theorem A with \cite[Théorème 1]{Re99}, we obtain that almost all RGD-systems of $2$-spherical type and rank $3$ are twin building lattices:

\medskip
\noindent \textbf{Corollary B:} Let $(W, S)$ a Coxeter system and let $\mathcal{D} = (G, (U_{\alpha})_{\alpha \in \Phi})$ be an RGD-system of type $(W, S)$. Assume that the following are satisfied:
\begin{itemize}
	\item $(W, S)$ is $2$-spherical of rank $3$ and $W$ is infinite.
	
	\item $G = \langle U_{\alpha} \mid \alpha \in \Phi \rangle$ and $\vert U_{\alpha} \vert < \infty$ for all $\alpha \in \Phi$.
\end{itemize}
Then $\mathcal{D}$ is a twin building lattice.

\medskip
\noindent \textbf{Corollary B (Kac-Moody version):} Let $(W, S)$ be a $2$-spherical Coxeter system of rank $3$ such that $W$ is infinite, and let $\mathbf{G}$ be the Kac-Moody group (in the sense of \cite{Ti87}) of type $(W, S)$. Then $\mathbf{G}(\FF_q)$ is a twin building lattice, where $\FF_q$ denotes the finite field with $q$ elements.

\medskip
Now the question is whether the finiteness still holds for some $q < n-1$. It turns out that in the class of Coxeter systems with complete Coxeter diagram this will not happen (cf.\ Theorem \ref{Theorem: divergence of growth function}):

\medskip
\noindent \textbf{Theorem C:} Let $(W, S)$ be a Coxeter system of rank $n\geq 3$ such that the underlying Coxeter diagram is the complete graph. Then $p_{(W, S)}\left( \frac{1}{n-2} \right) = \infty$.

\medskip
Suppose that the Coxeter diagram is $2$-spherical, but the Coxeter diagram is not the complete graph. If the number of non-edges in the Coxeter diagram compared to the number of edges is \emph{large}, then it is still possible that $p_{(W, S)} \left( \frac{1}{n-2} \right) <\infty$ holds (cf.\ \cite{Te15data}). We also remark that Theorem C can be used to exclude certain subdiagrams for twin building lattices, as parabolic subgroups of twin building lattices are again twin building lattices:

\medskip
\noindent \textbf{Corollary D:} Let $(W, S)$ be a Coxeter system, let $\mathcal{D}$ be an RGD-system of type $(W, S)$ with finite root groups and let $q_{\min} := \min\{ \vert U_{\alpha} \vert \mid \alpha \in \Phi \}$. If $\mathcal{D}$ is a twin building lattice, then there does not exist a subdiagram of $(W, S)$ with at least $q_{\min} +2$ vertices, whose underlying Coxeter diagram is the complete graph.

\renewcommand{\abstractname}{Acknowledgement}
\begin{abstract}
	I am very grateful to Bernhard M\"uhlherr for stimulating discussions and interesting questions on the topic. I also thank Corentin Bodart and Pierre-Emmanuel Caprace for valuable remarks on an earlier draft. I thank Corentin Borart for pointing out the reference \cite{AL02}.
\end{abstract}

\section{Preliminaries}

\subsection*{Growth of finitely generated groups}

This subsection is based on \cite{Te16}.

Let $G$ be a finitely generated group, and let $X = X^{-1} \subseteq G \backslash \{1\}$ be a finite, symmetric set of generators. The \emph{length} of $g\in G$ with respect to $X$ is the minimal $n$ such that $g = x_1 \cdots x_n$ with $x_i \in X$; the \emph{length function} will be denoted by $\ell_{(G, X)}: G \to \NN$. For $n\in \NN$, the \emph{sphere} in $\mathrm{Cay}(G, X)$ centered around $1_G$ with radius $n$ will be denoted by
\[ C_n^{(G, X)} := \left\{ g\in G \mid \ell_{(G, X)}(g) = n \right\}. \]
The cardinalities are defined as $c_n^{(G, X)} := \vert C_n^{(G, X)} \vert$. The \emph{growth function} of $(G, X)$ is given by
\[ p_{(G, X)}(t) := \sum_{n\geq 0} c_n^{(G, X)} t^n \in \ZZ[[t]]. \]

\subsection*{Coxeter systems}

Let $(W, S)$ be a Coxeter system and let $\ell := \ell_{(W, S)}$ be the corresponding length function. For $s, t \in S$ we denote the order of $st$ in $W$ by $m_{st}$. The \textit{Coxeter diagram} corresponding to $(W, S)$ is the labeled graph $(S, E(S))$, where $E(S) = \{ \{s, t \} \mid m_{st}>2 \}$ and where each edge $\{s,t\}$ is labeled by $m_{st}$ for all $s, t \in S$. The \textit{rank} of the Coxeter system is the cardinality of the set $S$.

It is well-known that for each $J \subseteq S$ the pair $(\langle J \rangle, J)$ is a Coxeter system (cf.\ \cite[Ch. IV, §$1$ Theorem $2$]{Bo68}). A subset $J \subseteq S$ is called \textit{spherical} if $\langle J \rangle$ is finite. The Coxeter system is called \textit{$2$-spherical} if $\langle J \rangle$ is finite for all $J \subseteq S$ containing at most $2$ elements (i.e.\ $m_{st} < \infty$ for all $s, t \in S$). Given a spherical subset $J$ of $S$, there exists a unique element of maximal length in $\langle J \rangle$, which we denote by $r_J$ (cf.\ \cite[Corollary $2.19$]{AB08}).

For $i \in \NN$ we define
\begin{itemize}[label=$\bullet$]
	\item $C_i := C_i^{(W, S)} = \{ w\in W \mid \ell(w) = i \}$ and $c_i := \vert C_i \vert = c_i^{(W, S)}$;
	
	\item $D_i := \{ w\in C_i \mid \exists! s\in S: \ell(ws) < \ell(w) \}$ and $d_i := \vert D_i \vert$;
\end{itemize}

\subsection*{The chamber system $\Sigma(W, S)$}

Let $(W, S)$ be a Coxeter system. Defining $w \sim_s w'$ if and only if $w^{-1}w' \in \langle s \rangle$ we obtain a chamber system with chamber set $W$ and equivalence relations $\sim_s$ for $s\in S$, which we denote by $\Sigma(W, S)$. We call two chambers $w, w'$ \textit{$s$-adjacent} if $w \sim_s w'$ and \textit{adjacent} if they are $s$-adjacent for some $s\in S$. A \textit{gallery of length $n$} from $w_0$ to $w_n$ is a sequence $(w_0, \ldots, w_n)$ of chambers where $w_i$ and $w_{i+1}$ are adjacent for each $0 \leq i < n$. A gallery $(w_0, \ldots, w_n)$ is called \textit{minimal} if there exists no gallery from $w_0$ to $w_n$ of length $k<n$ and we denote the length of a minimal gallery from $w_0$ to $w_n$ by $\ell(w_0, w_n)$. For $J \subseteq S$ we define the \textit{$J$-residue} of a chamber $c\in W$ to be the set $R_J(c) := c \langle J \rangle$. A \textit{residue} $R$ is a $J$-residue for some $J \subseteq S$; we call $J$ the \textit{type} of $R$ and the cardinality of $J$ is called the \textit{rank} of $R$. A residue is called \emph{spherical} if its type is a spherical subset of $S$. Let $R$ be a spherical $J$-residue. Two chambers $x, y \in R$ are called \emph{opposite in $R$} if $x^{-1} y = r_J$. Two residues $P, Q \subseteq R$ are called \emph{opposite in $R$} if for each $p\in P$ there exists $q\in Q$ such that $p, q$ are opposite in $R$. A \textit{panel} is a residue of rank $1$. It is a fact that for every chamber $x\in W$ and every residue $R$ there exists a unique chamber $z\in R$ such that $\ell(x, y) = \ell(x, z) + \ell(z, y)$ holds for each chamber $y\in R$. The chamber $z$ is called the \textit{projection} of $x$ onto $R$ and is denoted by $z = \proj_R x$.

A subset $\Sigma \subseteq W$ is called \emph{convex} if for any two chambers $c, d \in \Sigma$ and any minimal gallery $(c_0 = c, \ldots, c_k = d)$, we have $c_i \in \Sigma$ for all $0 \leq i \leq k$. Note that residues are convex by \cite[Example $5.44(b)$]{AB08}.

For two residues $R$ and $T$ we define $\proj_T R := \{ \proj_T r \mid r\in R \}$. By \cite[Lemma $5.36(2)$]{AB08} $\proj_T R$ is a residue contained in $T$. The residues $R$ and $T$ are called \emph{parallel} if $\proj_T R = T$ and $\proj_R T = R$.

\subsection*{Roots and walls}

Let $(W, S)$ be a Coxeter system. A \textit{reflection} is an element of $W$ that is conjugate to an element of $S$. For $s\in S$ we let $\alpha_s := \{ w\in W \mid \ell(sw) > \ell(w) \}$ be the \textit{simple root} corresponding to $s$. A \textit{root} is a subset $\alpha \subseteq W$ such that $\alpha = v\alpha_s$ for some $v\in W$ and $s\in S$. We denote the set of all roots by $\Phi(W, S)$. The set $\Phi(W, S)_+ := \{ \alpha \in \Phi(W, S) \mid 1_W \in \alpha \}$ is the set of all \textit{positive roots} and $\Phi(W, S)_- := \{ \alpha \in \Phi(W, S) \mid 1_W \notin \alpha \}$ is the set of all \textit{negative roots}. For each root $\alpha \in \Phi(W, S)$ we denote the \textit{opposite root} by $-\alpha$ and we denote the unique reflection which interchanges these two roots by $r_{\alpha}$. For $\alpha \in \Phi(W, S)$ we denote by $\partial \alpha$ (resp. $\partial^2 \alpha$) the set of all panels (resp. spherical residues of rank $2$) stabilized by $r_{\alpha}$. Furthermore, we define $\mathcal{C}(\partial \alpha) := \bigcup_{P \in \partial \alpha} P$ and $\mathcal{C}(\partial^2 \alpha) := \bigcup_{R \in \partial^2 \alpha} R$.

The set $\partial \alpha$ is called the \textit{wall} associated to $\alpha$. Let $G = (c_0, \ldots, c_k)$ be a gallery with $c_{i-1} \neq c_i$ for each $1 \leq i \leq k$. We say that $G$ \textit{crosses the wall $\partial \alpha$} if there exists $1 \leq i \leq k$ such that $\{ c_{i-1}, c_i \} \in \partial \alpha$. It is a basic fact that a minimal gallery crosses a wall at most once (cf.\ \cite[Lemma $3.69$]{AB08}). Moreover, a gallery which crosses each wall at most once is already minimal.

\begin{convention}
	For the rest of this paper we let $(W, S)$ be a Coxeter system of finite rank and we let $\Phi := \Phi(W, S)$ (resp. $\Phi_+ := \Phi(W, S)_+$ and $\Phi_- := \Phi(W, S)_-$).
\end{convention}

A pair $\{ \alpha, \beta \} \subseteq \Phi$ of roots is called \emph{prenilpotent}, if $\alpha \cap \beta \neq \emptyset \neq (-\alpha) \cap (-\beta)$. For a prenilpotent pair $\{ \alpha, \beta \}$ of roots we will write $\left[ \alpha, \beta \right] := \{ \gamma \in \Phi \mid \alpha \cap \beta \subseteq \gamma \text{ and } (-\alpha) \cap (-\beta) \subseteq (-\gamma) \}$ and $(\alpha, \beta) := \left[ \alpha, \beta \right] \backslash \{ \alpha, \beta \}$. We note that roots are convex (cf.\ \cite[Lemma $3.44$]{AB08}).

Let $(c_0, \ldots, c_k)$ and $(d_0 = c_0, \ldots, d_k = c_k)$ be two minimal galleries from $c_0$ to $c_k$ and let $\alpha \in \Phi$. Then $\partial \alpha$ is crossed by the minimal gallery $(c_0, \ldots, c_k)$ if and only if it is crossed by the minimal gallery $(d_0, \ldots, d_k)$.

\begin{lemma}\label{Lemma: residue and root}
	Let $R$ be a spherical residue of $\Sigma(W, S)$ of rank $2$ and let $\alpha \in \Phi$. Then exactly one of the following hold:
	\begin{enumerate}[label=(\alph*)]
		\item $R \subseteq \alpha$;
		
		\item $R \subseteq (-\alpha)$;
		
		\item $R \in \partial^2 \alpha$;
	\end{enumerate}
\end{lemma}
\begin{proof}
	It is clear that the three cases are exclusive. Suppose that $R \not\subseteq \alpha$ and $R \not\subseteq (-\alpha)$. Then there exist $c \in R \cap (-\alpha)$ and $d \in R \cap \alpha$. Let $(c_0 = c, \ldots, c_k = d)$ be a minimal gallery. As residues are convex, we have $c_i \in R$ for each $0 \leq i \leq k$. As $c\in (-\alpha), d\in \alpha$, there exists $1 \leq i \leq k$ with $c_{i-1} \in (-\alpha), c_i \in \alpha$. In particular, $\{ c_{i-1}, c_i \} \in \partial \alpha$ and hence $R \in \partial^2 \alpha$.
\end{proof}

\begin{lemma}\label{CM06Prop2.7}
	Let $R, T$ be two spherical residues of $\Sigma(W, S)$. Then the following are equivalent
	\begin{enumerate}[label=(\roman*)]
		\item $R, T$ are parallel;
		
		\item a reflection of $\Sigma(W, S)$ stabilizes $R$ if and only if it stabilizes $T$;
		
		\item there exist two sequences $R_0 = R, \ldots, R_n = T$ and $T_1, \ldots, T_n$ of residues of spherical type such that for each $1 \leq i \leq n$ the rank of $T_i$ is equal to $1+\mathrm{rank}(R)$, the residues $R_{i-1}, R_i$ are contained and opposite in $T_i$ and moreover, we have $\proj_{T_i} R = R_{i-1}$ and $\proj_{T_i} T = R_i$.
	\end{enumerate}
\end{lemma}
\begin{proof}
	This is \cite[Proposition $2.7$]{CM06}.
\end{proof}

\begin{lemma}\label{CM05Lem2.3}
	Let $\alpha \in \Phi$ be a root and let $x, y \in \alpha \cap \mathcal{C}(\partial \alpha)$. Then there exists a minimal gallery $(c_0 = x, \ldots, c_k = y)$ such that $c_i \in \mathcal{C}(\partial^2 \alpha)$ for each $0 \leq i \leq k$. Moreover, for each $1 \leq i \leq k$ there exists $L_i \in \partial^2 \alpha$ with $\{ c_{i-1}, c_i \} \subseteq L_i$.
\end{lemma}
\begin{proof}
	This is a consequence of \cite[Lemma $2.3$]{CM05} and its proof.
\end{proof}

\begin{lemma}\label{Lemma: residues in the 2-boundary are parallel}
	Let $\alpha, \beta \in \Phi, \alpha \neq \pm \beta$ be two roots and let $R, T \in \partial^2 \alpha \cap \partial^2 \beta$.
	\begin{enumerate}[label=(\alph*)]
		\item The residues $R$ and $T$ are parallel.
		
		\item If $\vert \langle J \rangle \vert = \infty$ holds for all $J \subseteq S$ containing three elements, then $R=T$.
	\end{enumerate}
\end{lemma}
\begin{proof}
	As $R, T \in \partial^2 \alpha \cap \partial^2 \beta$, there exist panels $P_1, Q_1 \in \partial \alpha$ and $P_2, Q_2 \in \partial \beta$ such that $P_1, P_2 \subseteq R$ and $Q_1, Q_2 \subseteq T$ (as in the proof of Lemma \ref{Lemma: residue and root}). By Lemma \ref{CM06Prop2.7} the panels $P_i, Q_i$ are parallel for both $i \in \{1, 2\}$. \cite[Lemma $17$]{DMVM11} yields that $P_i, \proj_T P_i$ are parallel and hence $\proj_T P_1 \in \partial \alpha, \proj_T P_2 \in \partial \beta$ by Lemma \ref{CM06Prop2.7}. As $\alpha \neq \pm \beta$, we deduce $\proj_T P_1 \neq \proj_T P_2$ and hence $\proj_T R$ contains the two different panels $\proj_T P_1$ and $\proj_T P_2$. In particular, $\proj_T R$ is not a panel. Since $\proj_T R$ is a residue contained in $T$, we deduce $\proj_T R = T$. Using similar arguments, we obtain $\proj_R T = R$ and $R, T$ are parallel. This proves $(a)$. Moreover, Lemma \ref{CM06Prop2.7} yields $R=T$, as there are no spherical residues of rank $3$ by assumption. This finishes the proof.
\end{proof}

\subsection*{Reflection and combinatorial triangles in $\Sigma(W, S)$}

A \textit{reflection triangle} is a set $T$ of three reflections such that the order of $tt'$ is finite for all $t, t' \in T$ and such that $\bigcap_{t\in T} \partial^2 \beta_t = \emptyset$, where $\beta_t$ is one of the two roots associated with the reflection $t$. Note that $\partial^2 \beta_t = \partial^2 (-\beta_t)$. A set of three roots $T$ is called \textit{combinatorial triangle} (or simply \textit{triangle}) if the following hold:

\begin{enumerate}[label=(CT\arabic*), leftmargin=*]
	\item The set $\{ r_{\alpha} \mid \alpha \in T \}$ is a reflection triangle.
	
	\item For each $\alpha \in T$, there exists $\sigma \in \partial^2 \beta \cap \partial^2 \gamma$ such that $\sigma \subseteq \alpha$, where $\{ \beta, \gamma \} = T \backslash \{ \alpha \}$.
\end{enumerate}

\begin{lemma}\label{reflectiontrianglechamber}
	Suppose that $(W, S)$ is $2$-spherical and the Coxeter diagram is the complete graph. If $T$ is a triangle, then $(-\alpha, \beta) = \emptyset$ holds for all $\alpha \neq \beta \in T$.
\end{lemma}
\begin{proof}
	This is \cite[Proposition $2.3$]{Bi22}.
\end{proof}

\begin{proposition}\label{auxres}
	Assume that $(W, S)$ is $2$-spherical and the Coxeter diagram is the complete graph. Let $R \neq T$ be two residues of rank $2$ such that $P := R \cap T$ is a panel. If $\ell(1_W, \proj_R 1_W) < \ell(1_W, \proj_T 1_W)$, then $\proj_T 1_W = \proj_P 1_W$.
\end{proposition}
\begin{proof}
	We let $\alpha \in \Phi_+$ be the root with $P \in \partial \alpha$. Let $(c_0 = 1_W, \ldots, c_{k'} = \proj_P c_0)$ be a minimal gallery with $c_k = \proj_R c_0$ for some $0 \leq k \leq k'$ and $c_k, \ldots, c_{k'} \in R$.
	
	We assume that $\proj_T c_0 \neq \proj_P c_0$ holds. Then we have $k' > \ell(1_W, \proj_T 1_W) > \ell(1_W, \proj_R 1_W) = k$. Let $(d_0 = 1_W, \ldots, d_{m'} = \proj_P d_0)$ be a minimal gallery with $d_m = \proj_T c_0$ for some $0 \leq m \leq m'$ and $d_m, \ldots, d_{m'} \in T$. We let $\beta \in \Phi_+$ be the root with $\{ d_m, d_{m+1} \} \in \partial \beta$ and we let $\gamma \in \Phi_+$ be the root with $\{ c_k, c_{k+1} \} \in \partial \gamma$. We will show that $\{ \alpha, -\beta, -\gamma \}$ is a triangle. Thus we first show that $\{ r_{\alpha}, r_{\beta}, r_{\gamma} \}$ is a reflection triangle. We have $T \in \partial^2 \alpha \cap \partial^2 \beta$ and, as a minimal gallery crosses a wall at most once, we deduce $\alpha \neq \beta$. Note that the wall $\partial \beta$ is crossed by the minimal gallery $(c_0, \ldots, c_{k'})$. Since $\partial^2 \alpha \ni R \neq T \in \partial^2 \alpha \cap \partial^2 \beta$ and $\alpha \neq \pm \beta$, Lemma \ref{Lemma: residues in the 2-boundary are parallel}$(b)$ implies $R \notin \partial^2 \beta$ and hence $\partial \beta$ is crossed by $(c_0, \ldots, c_k)$. As $k < k'$, we have $\proj_R 1_W \neq \proj_P 1_W$ and hence $\alpha \neq \gamma$. As $\alpha, \gamma \in \Phi_+$, we have $\alpha \neq \pm \gamma$.
	
	Assume that $o(r_{\beta} r_{\gamma}) = \infty$. We deduce $\beta \subseteq \gamma$. But $\partial\gamma$ has to be crossed by the gallery $(d_0, \ldots, d_{m'})$. Since $\partial^2 \alpha \ni T \neq R \in \partial^2 \alpha \cap \partial^2 \gamma$ and $\alpha \neq \pm \gamma$, we have $T \notin \partial \gamma^2$ by Lemma \ref{Lemma: residues in the 2-boundary are parallel}$(b)$ as before. This implies that $(d_0, \ldots, d_m)$ crosses the wall $\partial \beta$ and hence $\gamma \subseteq \beta$. This yields a contradiction and we have $o(r_{\beta} r_{\gamma}) <\infty$.
	
	As $R \in \partial^2 \alpha \cap \partial^2 \gamma$, Lemma \ref{Lemma: residues in the 2-boundary are parallel}$(b)$ implies $\partial^2 \alpha \cap \partial^2 \gamma = \{ R \}$. As $R \notin \partial^2 \beta$, we deduce $\partial^2 \alpha \cap \partial^2 \beta \cap \partial^2 \gamma = \emptyset$ and hence $\{ r_{\alpha}, r_{\beta}, r_{\gamma} \}$ is a reflection triangle.
	
	Now we have to verify (CT2). As $\partial^2 \gamma \not\ni T \in \partial^2 \alpha \cap \partial^2 \beta$ and $P \subseteq T \cap (-\gamma)$, we have $T \subseteq (-\gamma)$ by Lemma \ref{Lemma: residue and root}. As $\partial^2 \beta \not\ni R \in \partial^2 \alpha \cap \partial^2 \gamma$ and $P \subseteq R \cap (-\beta)$, we have $R \subseteq (-\beta)$. Let $1 \leq i \leq k$ be such that $\{ c_{i-1}, c_i \} \in \partial \beta$. Note that $\{ d_m, d_{m+1} \} \in \partial \beta, d_{m+1} \in (-\beta) \cap T \subseteq (-\gamma)$ and $c_i \in (-\beta) \cap \gamma$. By Lemma \ref{CM05Lem2.3} there exists a minimal gallery $(e_0 = d_{m+1}, \ldots, e_z = c_i)$ such that $e_j \in \C(\partial^2 \beta)$. As $d_{m+1} \in (-\gamma)$ and $c_i \in \gamma$, there exists $1 \leq p \leq z$ such that $e_{p-1} \in (-\gamma)$ and $e_p \in \gamma$. Again by Lemma \ref{CM05Lem2.3} there exists $L \in \partial^2 \beta$ such that $\{ e_{p-1}, e_p \} \subseteq L$, and hence $L \in \partial^2 \beta \cap \partial^2 \gamma$. As roots are convex and $e_0 = d_{m+1}, e_z = c_i \in \alpha$, we have $e_p \in L \cap \alpha$. As $\{ r_{\alpha}, r_{\beta}, r_{\gamma} \}$ is a reflection triangle (and hence $L \notin \partial^2 \alpha$), we obtain $L \subseteq \alpha$ by Lemma \ref{Lemma: residue and root}. This implies that $\{ \alpha, -\beta, -\gamma \}$ is a triangle and hence $(\alpha, \gamma) = \emptyset$ holds by Lemma \ref{reflectiontrianglechamber}. In particular, $k+1 = k'$ and $\ell(1_W, \proj_R 1_W) = \ell(1_W, \proj_P 1_W) -1 \geq \ell(1_W, \proj_T 1_W)$. This is a contradiction to the assumption and we conclude $\proj_T 1_W = \proj_P 1_W$.
\end{proof}

\begin{corollary}\label{Corollary: second up}
	Assume that $(W, S)$ is $2$-spherical and that the underlying Coxeter diagram is the complete graph. Suppose $w\in W$ and $s\neq t \in S$ with $\ell(ws) = \ell(w) +1 = \ell(wt)$ and suppose $w' \in \langle s, t \rangle$ with $\ell(w') \geq 2$. Then we have $\ell(ww'r) = \ell(w) + \ell(w') +1$ for each $r\in S \backslash \{ s, t \}$.
\end{corollary}
\begin{proof}
	Suppose $r\in S \backslash \{s, t\}$ and assume that $\ell(ww'r) = \ell(ww') -1$ holds for some $w' \in \langle s, t \rangle$ with $\ell(w') \geq 2$. Suppose $w'$ starts with $s$, i.e.\ $w' = sw''$ for some $w'' \in \langle s, t \rangle$ with $\ell(w'') = \ell(w') -1$. As $\ell(ww'r) = \ell(ww') -1$, one easily sees that $\ell(wstr) = \ell(wst) -1$ and $\ell(wsr) = \ell(ws) -1$ hold, too. We define $R := R_{\{r, t\}}(ws), T := R_{\{s, t\}}(w)$ and $P := R \cap T = \P_t(ws)$. Clearly, $\proj_T 1_W \neq \proj_P 1_W$. As $m_{rt} \geq 3$, we deduce $\ell(1_W, \proj_R 1_W) < \ell(1_W, \proj_T 1_W)$ and Proposition \ref{auxres} yields a contradiction.
\end{proof}

\begin{lemma}\label{Lemma: not both down}
	Assume that $(W, S)$ is $2$-spherical and that $m_{st} \geq 4$ holds for all $s\neq t \in S$. Suppose $w \in W$ and $s\neq t \in S$ with $\ell(ws) = \ell(w) +1 = \ell(wt)$. Then we have $\ell(w) +2 \in \{ \ell(wsr), \ell(wtr) \}$ for all $r\in S \backslash \{s, t\}$.
\end{lemma}
\begin{proof}
	Assume that $\ell(wsr) = \ell(w) = \ell(wtr)$. Then $\ell(wr) = \ell(w)-1$ and $\ell(wrs) = \ell(w) -2 = \ell(wrt)$. Let $R$ be the $\{ r, s \}$ residue containing $w$. As $m_{rs} \geq 4$, we deduce $\ell(wrsr) = \ell(wrs) -1$. Let $w' \in \langle s, t \rangle$ be such that $wr = (\proj_R 1_W) w'$. Then $\ell(w') \geq 2$ and the previous corollary implies $\ell(wrt) = \ell(wr) +1$, which is a contradiction. This finishes the proof.
\end{proof}

\begin{remark}
	Note that Lemma \ref{Lemma: not both down} is false without the assumption $m_{st} \geq 4$. To see this one can consider the Coxeter system of type $\tilde{A}_2$.
\end{remark}

\section{Some (in-)equalities}

To show the two main results (Theorem \ref{Theorem: convergence of growth function} and \ref{Theorem: divergence of growth function}), we will apply the quotient criterion. In order to do so we need a few inequalities, which we establish in this and the next section. We recall that for $i\in \NN$ we have
\begin{itemize}[label=$\bullet$]
	\item $C_i := \{ w\in W \mid \ell(w) = i \}$ and $c_i := \vert C_i \vert$;
	
	\item $D_i := \{ w\in C_i \mid \exists! s\in S: \ell(ws) < \ell(w) \}$ and $d_i := \vert D_i \vert$;
\end{itemize}

\begin{convention}
	In this section we assume that $(W, S)$ is of rank $n \geq 3$ and that there exists $m\geq 3$ such that $m_{st} = m$ holds for all $s\neq t \in S$. Moreover, we let $i>m$.
\end{convention}

\begin{remark}\label{Remark: unique}
	Note that $(W, S)$ is $2$-spherical and that the underlying Coxeter diagram is the complete graph. In particular, we have $\vert \langle J \rangle \vert = \infty$ for all $J \subseteq S$ containing three elements. This implies that for each $w\in W \backslash \{1_W\}$ there is either a unique element $s_w \in S$ with $\ell(w s_w) = \ell(w) -1$, or else there are exactly two elements $s_w \neq t_w \in S$ with $\ell(w s_w) = \ell(w) -1 = \ell(w t_w)$.
\end{remark}

\begin{lemma}\label{Lemma: ci-di = i-m}
	$c_i - d_i = \begin{pmatrix}
		n-2 \\ 2
	\end{pmatrix} c_{i-m} + (n-2) d_{i-m}$.
\end{lemma}
\begin{proof}
	Let $v \in C_i \backslash D_i$ be an element. Then there exist unique $s\neq t \in S$ with $\ell(vs) = \ell(v) -1 = \ell(vt)$. We define $R_v := R_{\{s, t\}}(v)$. Then we consider the mapping
	\[ f: C_i \backslash D_i \to C_{i-m}, v \mapsto \proj_{R_v} 1_W \]
	Note that $C_{i-m} = D_{i-m} \cup C_{i-m} \backslash D_{i-m}$. If $w\in C_{i-m} \backslash D_{i-m}$ is, there are exactly two elements in $S$, say $s_w \neq t_w \in S$ which decrease the length of $w$ (as $i>m$). Any other element $r\in S \backslash \{ s_w, t_w \}$ increases the length of $w$. For $n>3$ and $r_1 \neq r_2 \in S \backslash \{s_w, t_w\}$, we have $f(wr_{\{r_1, r_2\}}) = w$. For $n=3$ we have $f^{-1}(w) = \emptyset$. In both cases $w$ has $\begin{pmatrix}
		n-2 \\ 2
	\end{pmatrix}$ many preimages. If $w \in D_{i-m}$ is, there exists a unique $s_w \in S$ which decreases the length of $w$ and (similarly as before) $w$ has $\begin{pmatrix}
		n-1 \\ 2
	\end{pmatrix}$ many preimages. Note that $\begin{pmatrix}
		n-1 \\ 2
	\end{pmatrix} - \begin{pmatrix}
		n-2 \\ 2
	\end{pmatrix} = n-2$. We conclude:
	\allowdisplaybreaks
	\begin{align*}
		c_i - d_i = \vert C_i \backslash D_i \vert &= \sum_{w\in C_{i-m}} \vert f^{-1}(w) \vert \\
		&= \sum_{w\in C_{i-m} \backslash D_{i-m}} \vert f^{-1}(w) \vert + \sum_{w\in D_{i-m}} \vert f^{-1}(w) \vert \\
		&= \begin{pmatrix}
			n-2 \\ 2
		\end{pmatrix} \left( c_{i-m} - d_{i-m} \right) + \begin{pmatrix}
			n-1 \\ 2
		\end{pmatrix} d_{i-m} \\
		&= \begin{pmatrix}
			n-2 \\ 2
		\end{pmatrix} c_{i-m} + (n-2) d_{i-m}. \tag*{\qedhere}
	\end{align*}
\end{proof}

\begin{lemma}\label{Lemma: Mi}
	$2c_{i+1} - d_{i+1} = (n-2) c_i + d_i$.
\end{lemma}
\begin{proof}
	We put $M_i := \{ (w, s) \in C_i \times S \mid ws \in C_{i+1} \}$. We prove the claim by showing that both sides of the equation are equal to $\vert M_i \vert$.
	\begin{enumerate}[label=(\alph*)]
		\item $2c_{i+1} - d_{i+1} = \vert M_i \vert$: We consider the mapping
		\[ \pi: M_i \to C_{i+1}, (w, s) \mapsto ws. \]
		Clearly, $\pi$ is surjective. We define
		\allowdisplaybreaks
		\begin{align*}
			&C_{i+1}^1 := \{ w\in C_{i+1} \mid \vert \pi^{-1}(w) \vert = 1 \} &&\text{and} &&C_{i+1}^{>1} := \{ w\in C_{i+1} \mid \vert \pi^{-1}(w) \vert >1 \}.
		\end{align*}
		We show that $C_{i+1}^{>1} = C_{i+1} \backslash D_{i+1}$. Let $\bar{w} \in C_{i+1}^{>1}$ be an element. Then there exist $(w, s) \neq (w', s') \in \pi^{-1}(\bar{w})$. It follows $s\neq s'$ and hence $\bar{w} \in C_{i+1} \backslash D_{i+1}$. Now let $w\in C_{i+1} \backslash D_{i+1}$. Then there exist unique $s_w \neq t_w \in S$ which decrease the length of $w$. This implies $(ws_w, s_w) \neq (wt_w, t_w) \in \pi^{-1}(w)$. As $\vert \langle J \rangle \vert = \infty$ for all $J \subseteq S$ containing three elements, we deduce for every $1 \neq w\in W$ that
		\[ \vert \pi^{-1}(w) \vert \in \{ 1, 2 \}. \]
		We infer $C_{i+1}^1 = C_{i+1} \backslash C_{i+1}^{>1} = C_{i+1} \backslash \left( C_{i+1} \backslash D_{i+1} \right) = D_{i+1}$ and compute:
		\allowdisplaybreaks
		\begin{align*}
			\vert M_i \vert = \sum_{w\in C_{i+1}} \vert \pi^{-1}(w) \vert &= \sum_{w\in D_{i+1}} \vert \pi^{-1}(w) \vert + \sum_{w\in C_{i+1} \backslash D_{i+1}} \vert \pi^{-1}(w) \vert \\
			&= d_{i+1} + 2(c_{i+1} - d_{i+1}) \\
			&= 2c_{i+1} - d_{i+1}.
		\end{align*}
		
		\item $(n-2) c_i + d_i = \vert M_i \vert$: For a subset $T \subseteq C_i$, we define
		\allowdisplaybreaks
		\begin{align*}
			M_{i, T} &:= \{ (w, s) \in M_i \mid w \in T \}.
		\end{align*}
		For $w\in D_i$ there are exactly $n-1$ elements which increase the length of $w$. Thus we have $\vert M_{i, D_i} \vert = (n-1)d_i$. For $w\in C_i \backslash D_i$ there are exactly $n-2$ elements in $S$ which increase the length of $w$. Thus we have $\vert M_{i, C_i \backslash D_i} \vert = (n-2)(c_i - d_i)$. We conclude:
		\[ \vert M_i \vert = \vert M_{i, C_i \backslash D_i} \vert + \vert M_{i, D_i} \vert = (n-2) (c_i - d_i) + (n-1)d_i = (n-2) c_i + d_i. \tag*{\qedhere} \]
	\end{enumerate}
\end{proof}

\begin{lemma}\label{Lemma: inequalities}
	$c_{i+1} \leq (n-1) c_i - (n-2)d_{i-m+1} \leq (n-1) c_i$.
\end{lemma}
\begin{proof}
	The last inequality is obvious. Using Lemma \ref{Lemma: ci-di = i-m} and \ref{Lemma: Mi}, we deduce the following:
	\allowdisplaybreaks
	\begin{align*}
		c_{i+1} + (n-2) d_{i-m+1} &\leq 2c_{i+1} - d_{i+1} = (n-2) c_i + d_i \leq (n-1) c_i. \tag*{\qedhere}
	\end{align*}
\end{proof}

\begin{lemma}\label{Lemma: lower bound}
	Suppose $m>3$. Then the following hold:
	\begin{enumerate}[label=(\alph*)]
		\item $(n-2) c_i \leq c_{i+1}$;
		
		\item $(n-2) d_i \leq d_{i+1}$;
	\end{enumerate}
\end{lemma}
\begin{proof}
	We define $N_i := \{ (w, s) \in C_i \times S \mid ws \in D_{i+1} \}$. Then $N_i \to D_{i+1}, (w, s) \mapsto ws$ is a bijection and hence $\vert N_i \vert = d_{i+1}$. As in the proof of Lemma \ref{Lemma: Mi} we define for a subset $T \subseteq C_i$:
	\[ N_{i, T} := \{ (w, s) \in N_i \mid w\in T \}. \]
	We see that $c_{i+1} \geq d_{i+1} = \vert N_i \vert = \vert N_{i, D_i} \vert + \vert N_{i, C_i \backslash D_i} \vert$. Let $w\in C_i$. We now count pairs $(w, s) \in N_i$. We distinguish the following two cases:
	\begin{enumerate}[label=(\roman*)]
		\item $w\in D_i$: Let $s_w \in S$ be the unique element with $\ell(ws_w) < \ell(w)$. Let $t\in S \backslash \{ s_w \}$. Then $wt \in C_{i+1}$. Suppose $wt \notin D_{i+1}$. Then there exists $t\neq r \in S$ with $\ell(wtr) < \ell(wt)$. This implies $\ell(wr) < \ell(w)$ and the uniqueness of $s_w$ yields $r = s_w$. Now let $r \in S \backslash \{s_w, t\}$. Then $wr \in C_{i+1}$. Again, if $wr \notin D_{i+1}$, then $s_w$ would decrease the length of $wr$. But this is a contradiction to Lemma \ref{Lemma: not both down}. This implies $(w, r) \in N_{i, D_i}$ for all $r\in S \backslash \{s_w, t\}$. This shows $(b)$.
		
		\item $w\in C_i \backslash D_i$: Let $s_w \neq t_w \in S$ be the two elements with $\ell(w s_w) = \ell(w t_w) < \ell(w)$. Now let $r\in S \backslash \{ s_w, t_w \}$. Then $wr \in C_{i+1}$. We assume by contrary $wr \notin D_{i+1}$. Then there would exist $u \in S \backslash \{ r \}$ with $\ell(wr u) = \ell(w)$ and hence $\ell(wu) < \ell(w)$. By the uniqueness of $s_w$ and $t_w$ we obtain $u \in \{ s_w, t_w \}$. But then we obtain a contradiction to Corollary \ref{Corollary: second up}. We conclude $(w, r) \in N_{i, C_i \backslash D_i}$.
	\end{enumerate}
	We infer $c_{i+1} \geq \vert N_{i, D_i} \vert + \vert N_{i, C_i \backslash D_i} \vert \geq (n-2) d_i + (n-2) (c_i - d_i) = (n-2) c_i$.
\end{proof}

\section{Main results}

\subsection*{Reduction step}

Let $(W, S)$ and $(W', S')$ be two Coxeter systems. Following \cite{Te16}, we define $(W, S) \preceq (W', S')$ if there exists an injective map $\phi: S \to S'$ satisfying $m_{st} \leq m'_{\phi(s)\phi(t)}$ for all $s, t\in S$.

\begin{theorem}\label{Theorem: Theorem A in Te16}
	Let $(W, S)$ and $(W', S')$ be two Coxeter systems and let $a_n := a_n^{(W, S)}$ and $a_n' := a_n^{(W', S')}$. If $(W, S) \preceq (W', S')$, then $a_n \leq a_n'$.
\end{theorem}
\begin{proof}
	This is \cite[Theorem A]{Te16}.
\end{proof}

\subsection*{Convergence}

\begin{lemma}\label{Lemma: Existence of k}
	Let $(W, S)$ be of rank $n \geq 3$ and assume that there exists $m \geq 4$ such that $m_{st} = m$ holds for all $s\neq t \in S$. Then there exists $k \in \RR$ such that $\frac{d_i}{c_i} \geq k >0$ holds for all $i>m$.
\end{lemma}
\begin{proof}
	Using Lemma \ref{Lemma: ci-di = i-m}, \ref{Lemma: lower bound} and \ref{Lemma: lower bound}$(b)$, we compute:
	\allowdisplaybreaks
	\begin{align*}
		1 = \frac{c_i -d_i +d_i}{c_i} &= \frac{1}{c_i}\left( \begin{pmatrix}
			n-2 \\ 2
		\end{pmatrix}c_{i-m} + (n-2)d_{i-m} + d_i \right) \\
		&\leq \frac{1}{c_i}\left( \begin{pmatrix}
			n-2 \\ 2
		\end{pmatrix} \frac{1}{(n-2)^m} c_i + \left( \frac{1}{(n-2)^{m-1}} +1 \right) d_i \right) \\
		&= \frac{1}{c_i}\left( \frac{(n-3)}{2(n-2)^{m-1}} c_i + \left( \frac{1}{(n-2)^{m-1}} +1 \right) d_i \right) \\
		&\leq \frac{1}{2 (n-2)^{m-2}} + \left( \frac{1}{(n-2)^{m-1}} +1 \right) \frac{d_i}{c_i}
	\end{align*}
	We put
	\[ k := \left( 1 - \frac{1}{2 (n-2)^{m-2}} \right) \cdot \left( \frac{1}{(n-2)^{m-1}} + 1 \right)^{-1}. \]
	As $n \geq 3$ and $m \geq 4$, we have $k >0$. This proves the claim.
\end{proof}

\begin{theorem}\label{Theorem: convergence of growth function}
	Let $(W, S)$ be $2$-spherical and of rank $n \geq 3$. Then $p_{(W, S)}\left(\frac{1}{n-1}\right) < \infty$.
\end{theorem}
\begin{proof}
	Let $m := \max \{ 4, m_{st} \mid s, t \in S \}$ and let $(W', S')$ be the Coxeter system of rank $n$ with $m_{st}' = m$ for all $s\neq t \in S'$. Using Theorem \ref{Theorem: Theorem A in Te16} it suffices to show that
	\[ p_{(W', S')}\left( \frac{1}{n-1} \right) < \infty. \]
	By Lemma \ref{Lemma: Existence of k} there exists $k\in \RR$ such that $\frac{d_i}{c_i} \geq k >0$ holds for all $i >m$. We apply the quotient criterion. We use Lemma \ref{Lemma: inequalities} and compute for $i>2m-1$ and $t = \frac{1}{n-1}$:
	\allowdisplaybreaks
	\begin{align*}
		\frac{c_{i+1}t^{i+1}}{c_i t^i} \leq \frac{(n-1) c_i - (n-2)d_{i-m+1}}{(n-1) c_i} \leq 1 - \frac{(n-2) d_{i-m+1}}{(n-1)^m c_{i-m+1}} \leq 1- \frac{n-2}{(n-1)^m} k <1. \tag*{\qedhere}
	\end{align*}
\end{proof}

\subsection*{Divergence}

In this subsection we prove that the new lower bound $\frac{1}{n-1}$ for the finiteness of the growth function is optimal for the class of $2$-spherical Coxeter systems with complete Coxeter diagram.

\begin{lemma}\label{Lemma: ci leq di+di+1}
	Let $(W, S)$ be $2$-spherical and of rank $n \geq 4$ and assume that the underlying Coxeter diagram is the complete graph. Then $(n-2) c_i \leq d_i + d_{i+1}$.
\end{lemma}
\begin{proof}
	For $i=0$ we have $c_0 = 1$, $d_0 = 0$ and $d_1 = n$ and the claim follows. Thus we can assume $i>0$. As in Lemma \ref{Lemma: lower bound} we define $N_i := \{ (w, s) \in C_i \times S \mid ws \in D_{i+1} \}$ as well as $N_{i, T} := \{ (w, s) \in N_i \mid w\in T \}$ for $T \subseteq C_i$. We consider the mapping 
	\[ \pi: N_i \to D_{i+1}, (w,s) \mapsto ws. \]
	As before, $\pi$ is a bijection and we have $\vert N_i \vert = d_{i+1}$. Moreover, we have $N_i = N_{i, D_i} \cup N_{i, C_i \backslash D_i}$ and this union is disjoint. We now count pairs $(w, s)$ in $N_i$.
	
	We fix $w\in D_i$ and we let $s_w \in S$ be the unique element with $\ell(w s_w) = \ell(w) -1$. Assume that there are $r, s, t \in S \backslash \{ s_w \}$ pairwise distinct with $wr, ws, wt \in C_{i+1} \backslash D_{i+1}$. Similarly as in Lemma \ref{Lemma: lower bound}$(b)$ we deduce $\ell(wz s_w) = \ell(w)$ for each $z\in \{r, s, t\}$. As $m_{pq} \geq 3$ holds for all $p \neq q \in S$, we infer $\ell(w s_w z) = \ell(w s_w) -1$. As $\{ r, s, t \}$ is not spherical, this is a contradiction and we have for a fixed $w\in D_i$ at least $n-3$ tuples $(w, s)$ in $N_i$. 
	
	We fix $w\in C_i \backslash D_i$ and we let $s_w \neq t_w \in S$ be the unique elements with $\ell(w s_w) = \ell(w)- 1 = \ell(w t_w)$. Assume that there is $s\in S \backslash \{ s_w, t_w \}$ with $ws \in C_{i+1} \backslash D_{i+1}$. Then $\ell(w) \in \{ \ell(wss_w), \ell(wst_w) \}$. W.l.o.g.\ we assume $\ell(wss_w) = \ell(w)$. But then Corollary \ref{Corollary: second up} implies $\ell(wt_w) = \ell(w) +1$, which is a contradiction. Thus we have for a fixed $w\in C_i \backslash D_i$ exactly $n-2$ tuples $(w, s)$ in $N_i$. This implies that $(n-2) c_i - d_i = (n-3) d_i + (n-2) (c_i - d_i) \leq d_{i+1}$.
\end{proof}

\begin{theorem}\label{Theorem: divergence of growth function}
	Let $(W, S)$ be of rank $n \geq 4$ and assume that the underlying Coxeter diagram is the complete graph. Then $p_{(W, S)}\left( \frac{1}{n-2} \right) = \infty$.
\end{theorem}
\begin{proof}
	Let $(W', S')$ be the Coxeter system of rank $n$ with $m_{st}' = 3$ for all $s\neq t \in S'$. Using Theorem \ref{Theorem: Theorem A in Te16} it suffices to show that
	\[ p_{(W', S')} \left( \frac{1}{n-2} \right) = \infty. \]
	As before, we apply the quotient criterion. Using Lemma \ref{Lemma: Mi} and \ref{Lemma: ci leq di+di+1}, we deduce the following for $i >m = 3$ and $t = \frac{1}{n-2}$:
	\allowdisplaybreaks
	\begin{align*}
		\frac{c_{i+1} t^{i+1}}{c_i t^i} &= \frac{(n-2) c_i + d_i + d_{i+1}}{2 (n-2) c_i} = \frac{1}{2} + \frac{d_i + d_{i+1}}{2(n-2)c_i} \geq \frac{1}{2} + \frac{1}{2} = 1. \tag*{\qedhere}
	\end{align*}
\end{proof}

\bibliography{references}

\begin{thebibliography}{10}

\bibitem{AB08}
P.~Abramenko and K.~S. Brown.
\newblock {\em Buildings}, volume 248 of {\em Graduate Texts in Mathematics}.
\newblock Springer, New York, 2008.
\newblock Theory and applications.

\bibitem{AL02}
G.~N. Arzhantseva and I.~G. Lysenok.
\newblock Growth tightness for word hyperbolic groups.
\newblock {\em Math. Z.}, 241(3):597--611, 2002.

\bibitem{BS06}
U.~Bader and Y.~Shalom.
\newblock Factor and normal subgroup theorems for lattices in products of
  groups.
\newblock {\em Invent. Math.}, 163(2):415--454, 2006.

\bibitem{Bi22}
S.~Bischof.
\newblock On commutator relations in 2-spherical {RGD}-systems.
\newblock {\em Comm. Algebra}, 50(2):751--769, 2022.

\bibitem{BiDiss}
S.~Bischof.
\newblock {\em Construction of {RGD}-systems of type $(4,4,4)$ over
  $\mathbb{F}_2$}.
\newblock PhD thesis, Justus-Liebig-Universität Giessen, 2023.

\bibitem{Bo68}
N.~Bourbaki.
\newblock {\em Lie groups and {L}ie algebras. {C}hapters 4--6}.
\newblock Elements of Mathematics (Berlin). Springer-Verlag, Berlin, 2002.
\newblock Translated from the 1968 French original by Andrew Pressley.

\bibitem{CM05}
P.-E. Caprace and B.~M\"{u}hlherr.
\newblock Reflection triangles in {C}oxeter groups and biautomaticity.
\newblock {\em J. Group Theory}, 8(4):467--489, 2005.

\bibitem{CM06}
P.-E. Caprace and B.~M\"{u}hlherr.
\newblock Isomorphisms of {K}ac-{M}oody groups which preserve bounded
  subgroups.
\newblock {\em Adv. Math.}, 206(1):250--278, 2006.

\bibitem{CR09}
P.-E. Caprace and B.~R\'{e}my.
\newblock Simplicity and superrigidity of twin building lattices.
\newblock {\em Invent. Math.}, 176(1):169--221, 2009.

\bibitem{CR16}
P.-E. Caprace and B.~R\'{e}my.
\newblock Simplicity of twin tree lattices with non-trivial communication
  relations.
\newblock In {\em Topology and geometric group theory}, volume 184 of {\em
  Springer Proc. Math. Stat.}, pages 143--151. Springer, [Cham], 2016.

\bibitem{CG99}
L.~Carbone and H.~Garland.
\newblock Lattices in {K}ac-{M}oody groups.
\newblock {\em Math. Res. Lett.}, 6(3-4):439--447, 1999.

\bibitem{dlH00}
P.~de~la Harpe.
\newblock {\em Topics in geometric group theory}.
\newblock Chicago Lectures in Mathematics. University of Chicago Press,
  Chicago, IL, 2000.

\bibitem{DMVM11}
A.~Devillers, B.~M\"{u}hlherr, and H.~Van~Maldeghem.
\newblock Codistances of 3-spherical buildings.
\newblock {\em Math. Ann.}, 354(1):297--329, 2012.

\bibitem{Ma91}
G.~A. Margulis.
\newblock {\em Discrete subgroups of semisimple {L}ie groups}, volume~17 of
  {\em Ergebnisse der Mathematik und ihrer Grenzgebiete (3) [Results in
  Mathematics and Related Areas (3)]}.
\newblock Springer-Verlag, Berlin, 1991.

\bibitem{Re99}
B.~R\'{e}my.
\newblock Construction de r\'{e}seaux en th\'{e}orie de {K}ac-{M}oody.
\newblock {\em C. R. Acad. Sci. Paris S\'{e}r. I Math.}, 329(6):475--478, 1999.

\bibitem{Re05}
B.~R\'{e}my.
\newblock Integrability of induction cocycles for {K}ac-{M}oody groups.
\newblock {\em Math. Ann.}, 333(1):29--43, 2005.

\bibitem{Te15data}
T.~Terragni.
\newblock Data about hyperbolic coxeter systems, 2015.

\bibitem{Te16}
T.~Terragni.
\newblock On the growth of a {C}oxeter group.
\newblock {\em Groups Geom. Dyn.}, 10(2):601--618, 2016.

\bibitem{Ti87}
J.~Tits.
\newblock Uniqueness and presentation of {K}ac-{M}oody groups over fields.
\newblock {\em J. Algebra}, 105(2):542--573, 1987.

\end{thebibliography}
\bibliographystyle{abbrv}

\end{document}